\documentclass{amsart}
\usepackage{amsmath, amssymb, amsthm}
\usepackage{mathtools}
\usepackage{amscd}
\usepackage{graphicx} 
\usepackage{enumerate}
\usepackage{todonotes}
\usepackage{tabularx}
\usepackage{nicefrac}
\usepackage{framed}
\usepackage{xcolor}
\colorlet{shadecolor}{green!25}
\usepackage{url}

\newtheorem{thm}{Theorem}
\newtheorem{conj}{Conjecture}
\newtheorem{prop}{Proposition}
\newtheorem{lemma}{Lemma}

\newtheorem{cor}{Corollary}

\newtheorem*{thm*}{Theorem}
\newtheorem*{alg*}{Algorithm}
\newtheorem*{lemma*}{Lemma}

\theoremstyle{remark}
\newtheorem{rmk}{Remark}
\newtheorem*{rmk*}{Remark}

\newtheorem*{notation*}{Notation}
\newtheorem{example}[thm]{Example}
\newtheorem*{example*}{Example}

\theoremstyle{definition}

\newtheorem*{defn*}{Definition}

\newcommand{\NN}{\mathbb{N}}
\newcommand{\ZZ}{\mathbb{Z}}

\newcommand{\mybf}{\mathbb}

\newcommand{\bC}{\mybf{C}}
\newcommand{\bN}{\mybf{N}}
\newcommand{\bZ}{\mybf{Z}}
\newcommand{\bQ}{\mybf{Q}}

\newcommand{\cO}{\mathcal{O}}

\newcommand{\al}{\alpha}

\providecommand{\abs}[1]{\lvert#1\rvert}

\newcommand{\ra}{\rightarrow}

\DeclareMathOperator{\id}{id}

\title{On the behavior of Mahler's measure under iteration}
\author[Fili]{Paul Fili}
\author[Pottmeyer]{Lukas Pottmeyer}
\author[Zhang]{Mingming Zhang}

\email{paul.fili@okstate.edu}
\address{Oklahoma State University, Stillwater, OK, USA}
\address{Universit\"{a}t Duisberg-Essen, 45117 Essen, Germany}

\date{\today}

\begin{document}

\begin{abstract}
For an algebraic number $\alpha$ we denote by $M(\alpha)$ the Mahler measure of $\alpha$. As $M(\alpha)$ is again an algebraic number (indeed, an algebraic integer), $M(\cdot)$ is a self-map on $\overline{\mathbb{Q}}$, and therefore defines a dynamical system. The \emph{orbit size} of $\al$, denoted $\# \cO_M(\alpha)$, is the cardinality of the forward orbit of $\alpha$ under $M$. We prove that for every degree at least 3 and every non-unit norm, there exist algebraic numbers of every orbit size. We then prove that for algebraic units of degree 4, the orbit size must be 1, 2, or infinity. We also show that there exist algebraic units of larger degree with arbitrarily large but finite orbit size.
\end{abstract}

\maketitle

\section{Introduction}

The \emph{Mahler measure} of an algebraic number $\al$ with minimal polynomial $f(x) = a_n x^n + \cdots +a_0\in \bZ[x]$ is defined as:
\[
 M(\al) = \abs{a_n} \prod_{i=1}^n \max\{1, \abs{\al_i}\} = \pm a_n \prod_{\substack{i=1\\ \abs{\al_i}>1}}^n \al_i.
\]
where $f(x) = a_n \prod_{i=1}^n (x-\al_i)\in \bC[x]$. It is clear that $M(\al)\geq 1$ is a real algebraic integer, and it follows from Kronecker's theorem that $M(\al)=1$ if and only if $\al$ is a root of unity. Moreover, we will freely use the facts that $M(\alpha)=M(\beta)$ whenever $\alpha$ and $\beta$ have the same minimal polynomial, and that $M(\alpha)=M(\alpha^{-1})$. D.H. Lehmer \cite{Lehmer} famously asked in 1933 if the Mahler measure for an algebraic number which is not a root of unity can be arbitrarily close to $1$. This question became known as \emph{Lehmer's problem}, and (somewhat inaccurately) the statement that an absolute constant $c>1$ exists such that $M(\al)>1$ implies $M(\al)\geq c$ became known as \emph{Lehmer's conjecture}, despite the fact that Lehmer himself did not conjecture this and merely asked if one could find smaller values of the Mahler measure than he found. It is often suggested that the minimal value of $c$ is a Salem number, namely $\tau = 1.17\ldots$, which is the largest real root of the polynomial $f(x) = x^{10}+x^9-x^7-x^6-x^5-x^4-x^3+x+1$, discovered by Lehmer in his 1933 paper.

Although there has been much computational work performed in order to find irreducible polynomials of small Mahler measure (we refer the reader to M. Mossinghoff's website \cite{MossWeb} for the latest tables of known polynomials, as well as the papers by Mossinghoff \cite{Moss98} and Mossinghoff, Rhin, and Wu \cite{MossRhinWu}), remarkably, no polynomial of smaller nontrivial Mahler measure has been found since Lehmer's original 1933 work. Since that time, the best asymptotic bound towards Lehmer's problem was discovered by Dobrowolski \cite{Dob}. It is clear that in considering the problem, one can reduce to considering the Mahler measure of algebraic units. Smyth \cite{Smy} proved that in fact, non-reciprocal units have a minimal Mahler measure $\theta_0=M(\theta_0)$, where $\theta_0$ is the smallest Pisot-Vijayaraghavan number and is given by the positive root of $x^3-x-1$. In another direction, Borwein, Dobrowolski and Mossinghoff proved the Lehmer conjecture for polynomials with only odd coefficients \cite{BDM}.

The study of iteration of the Mahler measure began with questions about which algebraic numbers are themselves Mahler measures. Adler and Marcus \cite{AdlerMarcus} proved that every Mahler measure is a Perron number and asked if the Perron numbers given by the positive roots of $x^n - x - 1$ are also values of the Mahler measure for any $n > 3$. Recall that $\alpha$ is a Peron number if and only $\alpha>1$ is a real algebraic integer such that all conjugates of $\alpha$ over $\mathbb{Q}$ have absolute value $<\alpha$. This notion of `Perron number' was introduced by Lind \cite{Lind83} who also proved several properties of the class of Perron numbers in \cite{Lind84}, including that they are closed under addition and multiplication and are dense in the real interval $[1,\infty)$. Boyd \cite{BoydPerron} proved that the positive roots of $x^n - x - 1$ for $n>3$ were not values of the Mahler measure, but Dubickas \cite{DubickasNrsMM} showed that for every Perron number $\beta$, there exists a natural number $n$ such that $n\beta$ is a value of the Mahler measure. Dixon and Dubickas \cite{DixonDubickas} and Dubickas \cite{DubickasSalemNonrecip} established further results on which numbers are in the value set of $M$. However, the question whether a given number is a Mahler measure of an algebraic number is very hard to answer in general. For instance, it is an open question of A. Schinzel in \cite{Schinzel17+1} whether or not $\sqrt{17}+1$ is the Mahler measure of an algebraic number.

Dubickas \cite{DubickasClose} appears to have been the first to pose questions on the Mahler measure as a dynamical system, introducing the concept of the \emph{stopping time} of an algebraic number under $M$, defined as the number of iterations required to reach a fixed point. We note that the stopping time is one less than the cardinality of the forward orbit of the number under iteration of $M$, which we will call the \emph{orbit size}. Specifically, we set $M^0(\al)=\al$ and let $M^n(\al) = M\circ \cdots \circ M(\al)$ denote the $n$th iteration of $M$. We define the \emph{orbit of $\al$ under $M$} to be the set:
\begin{equation}
 \cO_M(\al) = \{ M^n(\al) : n \geq 0\}.
\end{equation}
Then the orbit size of $\al$ is $\# \cO_M(\al)$, while the stopping time is $\#\cO_M(\al)-1$. It is easy to see that for any algebraic number $\alpha$, $M(\alpha)\leq M^2(\alpha)$, so $M$ is nondecreasing after at least one iteration, and thus, the Mahler measure either grows, or is fixed. 

In fact, by Northcott's theorem, it is easy to see that if $\al$ is a wandering point of $M$, then $M^n(\al)\ra\infty$, as the degree of $M^n(\al)$ can never be larger than the degree of the Galois closure of the field $\bQ(\al)$. In particular, there are no cycles of length greater than $1$; each number $\al$ either wanders (that is, the orbit under $M$ is infinite), or it is preperiodic and ends in a fixed point of $M$. 
Dubickas claimed in \cite{DubickasClose} that `generically' $M^n(\al)\ra \infty$, however, he did not give an example or a proof of this. The first explicit results in this direction appear to have been by Zhang \cite{ZhangThesis}, who proved that if $[\bQ(\al):\bQ]\leq 3$, then $\# \cO_M(\al) <\infty$, and also found an algebraic number $\al$ of degree $4$ with minimal polynomial $x^4 + 5x^2 +x -1$ such that $M^{2n}(\al) = M^2(\al)^{2^{n-1}}$, proving that $M^n(\al)\ra \infty$ for this example. 

Further, it is trivial to see that the fixed points of $M$ correspond to natural numbers, Pisot-Vijayaraghavan numbers, and Salem numbers. This raises several natural questions: for example, can one show that the Lehmer problem could be reduced to the study of fixed points of $M$? The answer to such a question might help establish the long held folklore conjecture that Salem numbers are indeed minimal for Lehmer's problem.

Dubickas posed several questions in \cite{DubickasClose}, including whether one could classify all numbers of stopping time $1$ (that is, numbers which are not fixed by $M$, but for which $M(\al)$ is fixed), and whether algebraic numbers of arbitrary stopping time existed. In a later paper \cite{DubickasNrsMM}, he established, among other things, that for every $k\in \bN$, there exists a cubic algebraic integer of norm $2$ with stopping time $k$. 

In this paper, we will prove several other results regarding the stopping time of algebraic numbers. Our first result is a direct generalization of Dubickas's result:

\begin{thm}\label{thm:1}
For any $d \geq 3$, $l\in \ZZ\setminus\{\pm1,0\}$ and $k\in\NN$ there is an algebraic integer $\alpha$ of degree $d$, $N(\alpha)=l$ and $\# \cO_M(\alpha)=k$.
\end{thm}

The proof of Theorem \ref{thm:1} will be given in \S \ref{sec:non-units} below. To study the possible behaviour of algebraic units under itaration of $M$ is more delicate. It is clear that $\# \cO_M(\al)\leq 2$ for all algebraic units of degree at most $3$, and this result is (non-trivially) also true if the degree is $4$:

\begin{thm}\label{thm:degree4}
Let $\alpha$ be an algebraic unit of degree $4$. Then either $\# \cO_M(\alpha)\leq 2$ or $\# \cO_M(\alpha)=\infty$. Moreover, if $\# \cO_M(\alpha)=\infty$, then $M^{(3)}(\alpha)=M(\alpha)^2$.
\end{thm}

The first algebraic unit $\alpha$ with $\# \cO_M(\al)\geq 3$ we found has degree $6$ and orbit size $5$. It is given by any root of $x^6 - x^5 - 4x^4 - 2x^2 - 4x - 1$. Despite an extensive search, we did not find any unit of degree $5$ of orbit size $\geq 3$, nor a unit of degree $6$ of finite orbit size $\geq 6$.

It will follow from the proof of Theorem \ref{thm:degree4} that we have the following corollary:

\begin{cor}\label{cor:to-thm-degree4}
Let $\alpha$ be an algebraic unit of degree $4$, then the sequence $(\log M^{(n)}(\alpha))_{n\in\mathbb{N}}$ satisfies a linear homogeneous recursion.
\end{cor}

The proofs of Theorem \ref{thm:degree4} and Corollary \ref{cor:to-thm-degree4} are given in \S \ref{sec:degree4}. We note that, in the example of a degree 4 wandering point given by Zhang \cite{ZhangThesis}, the sequence $(\log M^{(n)}(\alpha))_{n\in\mathbb{N}}$ satisfied the recursion relation $x_n = 2x_{n-2}$ for $n \geq 3$. Based on the above corollary and further experimental data, we make the following conjecture:

\begin{conj}
For every algebraic unit $\alpha$, there exists a constant $k$ such that the sequence $(\log(M^{(n)}(\alpha)))_{n\geq k}$ satisfies a linear homogeneous recursion.
\end{conj}

We note that, in the case of a large Galois group, the behavior of units is particularly simple. We prove that, if the Galois group contains the alternating group, then the orbit of a unit must either stop after at most one iteration, or the unit wanders. Specifically, we prove in \S \ref{sec:Sd-or-Ad} the following theorem:

\begin{thm}\label{thm:Sd-or-Ad}
If $\alpha$ is an algebraic unit of degree $d$ such that the Galois group of the Galois closure of $\mathbb{Q}(\alpha)$ over $\mathbb{Q}$ contains the alternating group $A_d$, then $\# \cO_M(\alpha)\in\{1,2,\infty\}$.

More precisely, if $\al$ is as above, of degree $\geq 5$, and such that none of $\pm \al^{\pm 1}$ is conjugate to a Pisot number, then $\#\cO_M(\al)=\infty$. 
\end{thm}


One might be led by Theorems \ref{thm:degree4} and \ref{thm:Sd-or-Ad} to suspect that, in fact, algebraic units cannot have arbitrarily large but finite orbits under $M$. However, we prove that this is not the case. 

\begin{thm}\label{thm:deg-12}
Let $S\in \mathbb{N}$ be arbitrary, and let $d\geq 12$ be divisible by $4$. Then there exist algebraic units of degree $ d $ whose orbit size is finite but greater than $ S $.
\end{thm}

The proof is given in Section \ref{sec:deg-12}. It would be interesting to know whether there are large finite orbits of algebraic units in any degree less than $12$.

\section{Arbitrary orbit size for non-units and proof of Theorem \ref{thm:1}}\label{sec:non-units}

In \cite{DubickasNrsMM}, Dubickas proved the case $d=3$ and $l=2$ (and $k$ arbitrary). In order to prove Theorem \ref{thm:1}, we will start with a few examples.

\begin{example}
Since there are Pisot-Vijayaraghavan numbers of any degree and norm, we know that for any $d\in\NN$ and any $l\in\ZZ\setminus\{\pm1,0\}$ there are algebraic numbers $\alpha$ of degree $d$, norm $l$ and orbit size $1$. By Perron's criterion, we may take the largest root of $x^d + l^2 x^{d-1} +l$.

Similarly, the polynomial $x^d + l^{d}x +l$ has precisely one root $\beta$ inside the unit circle and all other roots are of absolute value $>\vert l\vert$. Hence, the polynomial is irreducible. Let $\alpha$ be the largest root of this polynomial. Then $M(\alpha)=\vert \frac{l}{\beta}\vert$, which is a Pisot number. Thus, $\alpha$ has norm $l$, degree $d$ and orbit size $2$.
\end{example}

\begin{example}
For any $l\in\ZZ\setminus\{\pm1,0\}$ we consider $f(x)=x^3 -l^2 x +l$. Let $\alpha_1,\alpha_2,\alpha_3$ be the roots of $f$ ordered such that $\vert \alpha_1 \vert \geq \vert \alpha_2 \vert \geq \vert \alpha_3 \vert$.

If $l \geq 2$ we have 
\[
\begin{array}{lcl}
f (-l-1) = -2l^2 -2l -1 < 0 & \quad & f(-l)=l >0 \\
f(l-1)=-2l^2 +4l -1 <0 & \quad & f(l)=l>0 \\
f(1)=1-l^2 +l < 0 & \quad & f(\frac{1}{l})=\frac{1}{l^3} >0
\end{array}
\]
Hence, the three roots are real and none of them is an integer. If $f$ is reducible, then one of the factors must be linear, this is a contradiction since $f$ is monic. Hence, $f$ is irreducible and it follows $\alpha_1 \in (-l-1,-l)$, $\alpha_2 \in (l-1,l)$ and $\alpha_3 \in (\frac{1}{l},1)$. Therefore we find $M^{(0)}(\alpha_1)=\alpha_1$, $M^{(1)}(\alpha_1)=-\alpha_1\alpha_2 =\frac{l}{\alpha_3}$, $M^{(2)}(\alpha_1)=M(\frac{l}{\alpha_3})=\frac{l^2}{\alpha_2\alpha_3}=-\alpha_1 l$, $M^{(3)}(\alpha_1)=M(-\alpha_1 l)=\alpha_1 l \alpha_2 l \alpha_3 l =l^4 \in\ZZ$. These are all elements in the orbit of $\alpha_1$ under iteration of $M$. Hence, $\alpha_1$ is an algebraic integer of degree $3$, $N(\alpha_1)=l$ and $\# \cO_M(\alpha_1)=4$. Moreover $-\alpha_1$ is an algebraic integer of degree $3$, $N(-\alpha_1)=-l$ and $\# \cO_M(-\alpha_1)=4$.

In the same fashion one can prove that any root of the polynomial $x^3 +lx^2 -l$ is of degree $3$, norm $-l$ and orbit size $3$.
\end{example}

\begin{example}
Again let $l\in\ZZ\setminus\{\pm1,0\}$ be arbitrary and consider $f(x)=x^4 -l^2 x^2 + (l^2 -l)x +l$. The four roots of $f$ are ordered as $\vert \alpha_1\vert \geq \vert \alpha_2 \vert \geq \vert \alpha_3\vert \geq \vert \alpha_4\vert$. A direct computation shows that $f$ is irreducible and $\# \cO_M(\alpha_1)=4$ if $l\in\{-3,-2,-4\}$. If $l \not\in\{-3,-2-1,0,1,2\}$, then we show as in the last example that
\[
\alpha_1 \in (-l-1,-l), \quad \alpha_2\in (l-1,l), \quad \alpha_3 \in (1,2), \quad \alpha_4 \in (-1,-\frac{1}{l^2})
\]
if $l>0$, and
\[
\alpha_1 \in (-l-1,-l), \quad \alpha_2\in (l-1,l), \quad \alpha_3 \in (1,2), \quad \alpha_4 \in (1,\frac{1}{l^2})
\]
if $l < 0$. Obviously $f$ has no linear factor. Moreover, $\alpha_4$ and $\alpha_1$ must be Galois conjugates, since the norm of $\alpha_1$ has to be a divisor of $l$. Hence, if $f$ is not irreducible it factors into $g(x)=(x-\alpha_1)(x-\alpha_4)$ and $h(x)=(x-\alpha_2)(x-\alpha_3)$. This can only occur if $g$ and $h$ are in $\ZZ[x]$. Comparing the size of the roots, the only possibilities are $g(x)=x^2 + (l+1)x +1$ and $h(x)=x^2 -(l+1)x +l$. However, multiplying these two polynomials does not give $f$. Hence, $f$ is irreducible.

Now we calculate the orbit size of $\alpha_1$. We have $M(\alpha_1)=-\frac{l}{\alpha_4}$, $M^{(2)}(\alpha_1)=\pm l^2 \alpha_1$, $M^{(3)}(\alpha_1)=\pm l^9$, and hence $\# \cO_M(\alpha_1)=4$. We have shown, that any root $\alpha$ of $f$ is an algebraic integer of degree $4$, norm $l$ and orbit size $4$.
\end{example}

\begin{example}
One can show with similar methods as above, that any root of $x^d - l^{d-2} x +l$ has orbit size $3$, for all $d\geq 4$ and $l\in\ZZ\setminus\{\pm1,0\}$: To this end, we note
\begin{equation}\label{eq:onerootinside}
\vert -l^{d-2} z \vert = \vert l \vert^{d-2} > \vert l\vert +1 \geq \vert z^d + l\vert \quad \forall ~ z\in\mathbb{C}, ~\vert z \vert =1,
\end{equation}
and 
\begin{equation}\label{eq:allrotslessl}
\vert z^d \vert = \vert l \vert^{d} > \vert l\vert^{d-1} + \vert l\vert \geq \vert -l^{d-2} z + l\vert \quad \forall ~ z\in\mathbb{C}, ~\vert z \vert =\vert l\vert.
\end{equation}
Now we apply Rouch\'e's theorem. Then \eqref{eq:onerootinside} tells us that $x^d - l^{d-2} x +l$ has precisely one root $\alpha_d$ inside the unit circle, and \eqref{eq:allrotslessl} tells us that all roots $\alpha_1,\ldots,\alpha_d$ of $x^d - l^{d-2} x +l$ have absolute value $<\vert l \vert$.

Before we proceed with calculating the orbit size of one of these roots, we need to show that $x^d - l^{d-2} x +l$ is irreducible. This is obviously the case if $\vert l\vert$ is a prime number. So in particular, we can assume that $\vert l \vert \geq 4$. Using this assumption and $d\geq 4$, the same calculation as in \eqref{eq:onerootinside} proves that there is precisely one root of $x^d - l^{d-2} x +l$ of absolute value $\leq \sqrt{\vert l \vert}$ (necessarily $\alpha_d$). 

It follows that no product of two or more of the elements $\alpha_1,\ldots,\alpha_{d-1}$ can be a divisor of $l$. Hence, the only possibility for $x^d - l^{d-2} x +l$ to be reducible is, if it has a root $a\in\mathbb{Z}$. This $a$ must be a divisor of $\vert l\vert$ and it must satisfy $a^d = l^{d-2}a-l$. Hence, $a^{d-1} \mid l$ which implies $\vert a \vert^{d-1} \leq \vert l\vert$. This is not possible, as we have just seen that $\vert a\vert \geq \sqrt{\vert l\vert}$. It follows that $x^d - l^{d-2} x +l$ is indeed irreducible, and $\alpha_1$ is an algebraic integer of degree $d$, and norm $l$.

We then have:
\begin{itemize}
\item $M^{(1)}(\alpha_1) = \alpha_1\cdots \alpha_{d-1} =\frac{l}{\vert \alpha_d \vert} \notin \mathbb{Z}$,
\item $M^{(2)}(\alpha_1)=M(\pm \frac{l}{\alpha_d})=\pm \prod_{i=1}^d \frac{l}{\alpha_i} \in \mathbb{Z}$, and
\item $M^{(n)}(\alpha_1) = M^{(2)}(\alpha_1)$ for all $n\geq 2$.
\end{itemize}
Hence $\alpha_1$ has orbit size $3$.
\end{example}

\begin{prop}\label{prop:assump}
Let $d \geq 3$ be an integer and let $\alpha_1,\ldots,\alpha_d$ be a full set of Galois conjugates of an algebraic integer $\alpha$. Assume the following conditions:
\begin{enumerate}[(i)]
\item $\vert \alpha_1 \vert > \vert \alpha_2 \vert \geq \ldots \geq \vert \alpha_{d-1} \vert > 1 > \vert \alpha_d \vert$,
\item $\vert \alpha_i \vert \leq \vert N(\alpha) \vert$ for all $i\in\{2,\ldots,d\}$, 
\end{enumerate}
Then $\alpha$ is a pre-periodic point of $M$. More precisely, if we let
\begin{multline*}
c(\alpha)= \min\{\min\{k\in\mathbb{N}: 2\mid k \text{ and } \vert \alpha_d \cdot N(\alpha)^{b_{k}} \vert > 1\},\\
\min\{k\in\mathbb{N}: 2\nmid k \text{ and } \vert \alpha_1 \vert < \vert N(\alpha)^{b_k} \vert \}\},
\end{multline*}
where we define $b_1=1$, and $b_n = b_{n-1}\cdot (d-1)+(-1)^{n-1}$ for all $n\geq 2$, then $\# \cO_M(\alpha)=c(\alpha)+2$.
\end{prop}
\begin{proof}
First we note, that $\alpha$ cannot be an algebraic unit. Hence, $\vert N(\alpha)\vert \geq 2$ and $b_k\geq 1$ for all $k$. We claim that $b_k \rightarrow \infty$. To see this, notice that $b_1=1, b_2=d-2\geq 1$, and we want to show that for $n\geq 3$, $b_n\geq (d-2)(d-1)^{n-2}+1$. Now, this is true for $n=3$, since $b_3=(d-2)(d-1)+1$. By induction, suppose $b_{n-1}\geq(d-2)(d-1)^{n-3}+1$, then $b_n\geq((d-2)(d-1)^{n-3}+1)(d-1)+(-1)^{n-1}=(d-2)(d-1)^{n-2}+(d-1)+(-1)^{n-1}\geq(d-2)(d-1)^{n-2}+1$, as desired. Therefore, $b_n\geq 1$ for all $n$, and $b_n \rightarrow \infty$.

So the integer $c:=c(\alpha)$ does indeed exist. We claim that for all $k \leq c$ we have 
\begin{equation}\label{eq:mahler1}
M^{(k)}(\alpha)=\begin{cases} \pm \frac{N(\alpha)^{b_k}}{\alpha_d} & \text{ if } 2 \nmid k \\ \pm N(\alpha)^{b_k} \cdot \alpha_1 & \text{ if } 2\mid k \end{cases}
\end{equation}
Note that $\alpha_1,\alpha_d \in \mathbb{R}$, since there is no other conjugate of the same absolute value. Therefore, the sign in \eqref{eq:mahler1} has to be chosen such that the value is positive. We prove the claim by induction.

For $k=1$, we calculate $M^{(1)}(\alpha)=M(\alpha) = \pm \alpha_1\cdot\ldots\cdot \alpha_{d-1}= \pm \frac{N(\alpha)}{\alpha_d}=\pm \frac{N(\alpha)^{b_1}}{\alpha_d}$, by assumption (i). Now assume, that \eqref{eq:mahler1} is correct for a fixed $k < c$. If $k$ is even, then by assumption (i) we have 
\begin{align*}
M^{(k+1)}(\alpha) & = M(\pm N(\alpha)^{b_k} \cdot \alpha_1) = \pm N(\alpha)^{b_k\cdot (d-1)}\cdot\alpha_1 \cdot \ldots\cdot \alpha_{d-1}\\  &= \pm \frac{N(\alpha)^{b_k \cdot (d-1) +1}}{\alpha_d}= \pm \frac{N(\alpha)^{b_{k+1}}}{\alpha_d}.
\end{align*}
Here we have used that $k<c$ and hence $\vert  N(\alpha)^{b_k} \cdot \alpha_{d} \vert <1$. 

If $k$ is odd, then by assumption (ii) we have
\begin{align*}
M^{(k+1)}(\alpha) & = M(\pm \frac{N(\alpha)^{b_k}}{\alpha_d}) = \pm \frac{N(\alpha)^{b_k}}{\alpha_d}\cdot \frac{N(\alpha)^{b_k}}{\alpha_{d-1}} \cdot \ldots \cdot \frac{N(\alpha)^{b_k}}{\alpha_2}\\ &= \pm \frac{N(\alpha)^{b_k\cdot (d-1)}}{\alpha_2\cdot\ldots\cdot\alpha_{d-1}} = \pm N(\alpha)^{b_k\cdot (d-1)-1}\cdot\alpha_1 = \pm N(\alpha)^{b_{k+1}}\cdot\alpha_1.
\end{align*}
Here we have used that $k<c$ and hence $\vert \frac{N(\alpha)^{b_k}}{\alpha_1} \vert < 1$.
This proves the claim. Moreover, the proof of the claim shows that $M^{(k+1)}(\alpha) > M^{(k)}(\alpha_1)$ for all $k \in \{0,\ldots,c-1\}$. 

Now, we calculate $M^{(c+1)}(\alpha)$. By definition of $c$, every conjugate of $M^{c}(\alpha)$ is greater than $1$ in absolute value. Therefore, $M^{c+1}(\alpha)\in\mathbb{N}$. It follows, that $M^{(c+2)}(\alpha)=M^{(c+1)}(\alpha)$. Hence, $\# \cO_M(\alpha_1)=c+2$ as claimed.
\end{proof}

It remains to prove the existence of an algebraic number of degree $d$ satisfying the assumptions of Proposition \ref{prop:assump} for an arbitrary $c$.

The strategy is as the following: We will prove the locations of the roots of a class of irreducible polynomials satisfying assumptions (i) and (ii) from Proposition \ref{prop:assump}, then by Proposition \ref{prop:assump}, show that any root of one of the polynomials in the class will have desired degree, norm and orbit size.

We fix for the rest of this section arbitrary integers $d\geq 3$, $c\geq 2$ and $l\in\ZZ\setminus\{\pm 1,0\}$. Moreover, we define
\[
f_n(x)=x\cdot (x^{d-2}-2)\cdot (x-n) + l
\]
and denote the roots of $f_n$ by $\alpha_1^{(n)},\ldots,\alpha_d^{(n)}$ ordered such that
\[
\vert \alpha_1^{(n)} \vert \geq \vert \alpha_2^{(n)} \vert \geq \ldots \geq \vert \alpha_d^{(n)}  \vert.
\]

\begin{lemma}\label{lem:rootslocation}
Let $n \geq \vert l \vert +3$ be an integer. With the notation from above we have $ \alpha_1^{(n)} \in (n-\frac{1}{n},n+ \frac{1}{n})$, $\alpha_d^{(n)} \in (-\frac{\vert l \vert}{n},-\frac{1}{2n})\cup(\frac{1}{2n},\frac{\vert l \vert}{n})$, and $\vert \alpha_i^{(n)} \vert \in (1, \sqrt[d-2]{3-\frac{1}{d}})$ for all $i\in \{2,\ldots,d-1\}$. Moreover, $\alpha_d^{(n)}$ is negative if and only if $\alpha_1^{(n)} <n$.
\end{lemma}
\begin{proof}
We apply Rouch\'e's theorem and first prove the location of $\alpha_1^{(n)}$. Let $z$ be any complex number with $\vert z \vert = n+\frac{1}{n}$. Then 
\begin{align*}
 & \vert z\cdot (z^{d-2} -2 ) \cdot (z-n)\vert \\ & \geq \left | n+\frac{1}{n}\right | \cdot \left\vert (n+\frac{1}{n})^{d-2}-2\right\vert \cdot \frac{1}{n} \\ & = \left\vert 1+\frac{1}{n^2}\right \vert \cdot \left\vert (n+\frac{1}{n})^{d-2}-2\right\vert \\ & > \vert l \vert
\end{align*}
Hence by Rouch\'e's theorem, $f_{n}$ has exactly as many roots of absolute value $< n+\frac{1}{n}$ as $x\cdot (x^{d-2}-2)\cdot (x-n)$, so $f_{n}$ has $d$ roots of absolute value $< n+\frac{1}{n}$.
Now, let $z$ be any complex number with $\vert z \vert = n - \frac{1}{n}$, suppose that $n=\vert l \vert + m$ where $m\geq 3$. Then 
\begin{align*}
& \left\vert z\cdot (z^{d-2} -2 ) \cdot (z-n)\right \vert \\ & \geq \left\vert n-\frac{1}{n}\right\vert \cdot \left\vert (n-\frac{1}{n})^{d-2}-2\right\vert \cdot \frac{1}{n} \\ & \geq \left\vert n-\frac{1}{n}\right\vert \cdot \left\vert (n-\frac{1}{n})-2\right\vert \cdot \frac{1}{n} \\ & = (\vert l \vert +m- \frac{1}{\vert l \vert +m})(\vert l \vert +m- \frac{1}{\vert l \vert +m}-2)\cdot \frac{1}{\vert l \vert +m} \\ & = (1-\frac{1}{(\vert l \vert +m)^2})(\vert l \vert - \frac{1}{\vert l \vert +m} +m-2) \\ & =\vert l \vert -\frac{\vert l \vert}{(\vert l \vert +m)^2}-\frac{1}{\vert l \vert +m}+(m-2)+\frac{1}{(\vert l \vert +m)^3}-\frac{m}{(\vert l \vert +m)^2}+\frac{2}{(\vert l \vert +m)^2}> \vert l \vert,
\end{align*}
since $m\geq 3$. Again by Rouch\'e's theorem, $f_n$ has $d-1$ roots of absolute value $<n-\frac{1}{n}$. Since $f_n$ has no roots on the circle $\vert z \vert = n- \frac{1}{n}$, $f_n$ has a single root in $(-n-\frac{1}{n}, -n+\frac{1}{n})\cup (n-\frac{1}{n}, n+\frac{1}{n})$. Now, 
 \begin{align*}
 &\left\vert (-n-\frac{1}{n})((-n-\frac{1}{n})^{d-2}-2)(-2n-\frac{1}{n}) \right\vert \\ & \geq(\vert l \vert+2)\left\vert(n+\frac{1}{n})^{d-2}-2\right\vert (2n+\frac{1}{n})\\ & \geq (\vert l \vert+2)\vert \vert l \vert (2(\vert l \vert +2)) \\ & \geq (\vert l \vert+2)\vert \vert l \vert (2\vert l \vert +4) \\ & \geq \vert l \vert^{2}>\vert l \vert.
\end{align*}
Similarly, \[ \left\vert (-n+\frac{1}{n})((-n+\frac{1}{n})^{d-2}-2)(-2n+\frac{1}{n}) \right\vert \geq 2\vert l \vert^2 > \vert l \vert.\]
Since \[(-n-\frac{1}{n})((-n-\frac{1}{n})^{d-2}-2)(-2n-\frac{1}{n})\] has the same sign as \[(-n+\frac{1}{n})((-n+\frac{1}{n})^{d-2}-2)(-2n+\frac{1}{n}),\] $f_{n}(-n+\frac{1}{n})$ has the same sign as $f_{n}(-n-\frac{1}{n})$. Therefore, since there is only one root in the annulus $\vert z \vert \in  (n-\frac{1}{n},n+ \frac{1}{n})$, which is necessarily real, $f_{n}$ cannot have any root in the interval $(-n-\frac{1}{n}, -n+\frac{1}{n})$, thus $f_n$ has a single root in the interval $(n-\frac{1}{n}, n+\frac{1}{n})$.

To prove the location of $\alpha_d^{(n)}$, let $z$ be any complex number with $\vert z \vert =\frac{\vert l \vert}{n}$. Then, 
\begin{align*}
 & \vert z\cdot (z^{d-2} -2 ) \cdot (z-n)\vert  \geq  \frac{\vert l \vert}{n} \cdot (2-\frac{\vert l \vert}{n})\cdot (n-\frac{\vert l \vert}{n}) \\  = & 2\vert l \vert - 2 \frac{\vert l \vert^2}{n^2} - \frac{\vert l \vert^2}{n}+\frac{\vert l \vert^3}{n^3} > 2\vert l \vert - 2 \frac{\vert l \vert^2}{n^2} - \frac{\vert l \vert^2}{n} \\ \geq & 2\vert l\vert - \vert l \vert \frac{\vert l \vert^2+4\vert l\vert}{(\vert l \vert +2)^2} > \vert l \vert.
\end{align*}
By Rouch\'e's theorem, $f_n$ has exactly as many roots of absolute value $<\frac{\vert l \vert}{n}$ as the polynomial $x\cdot(x^{d-2}-2)\cdot(x-n)$. This is, $f_n$ has exactly one root of absolute value $<\frac{\vert l \vert}{n}$. This root is necessarily real. A straightforward computation shows that $f_n(\pm\frac{1}{2n})$ have the same sign as $f_n(0)$. Hence $f_n$ cannot have any root in the interval $(-\frac{1}{2n},\frac{1}{2n})$.

To show the location of $\alpha_i^{(n)}$ for all $i\in \{2,\ldots,d-1\}$, let $z$ be any complex number with $\vert z \vert =1$. Then,
\begin{align*}
 & \vert z\cdot (z^{d-2} -2 ) \cdot (z-n)\vert  \\ & = \vert z^{d-2} -2 \vert \cdot \vert z-n\vert \\ & \geq n-1> \vert l \vert,
\end{align*}
so $f_n$ has a single root of absolute value $<1$. The argument above also shows that $f_n$ has no roots on the circle $\vert z \vert =1$.
Now, let $z$ be any complex number with $\vert z \vert =\sqrt[d-2]{3-\frac{1}{d}}$. Then, 
\begin{align*}
 & \vert z\cdot (z^{d-2} -2 ) \cdot (z-n)\vert  \\ & \geq (3-\frac{1}{d})^{\frac{1}{d-2}}\cdot(1-\frac{1}{d})\cdot(n-(3-\frac{1}{d})^{\frac{1}{d-2}}).
\end{align*}
Notice that since $n\geq\vert l \vert +3$, $n-(3-\frac{1}{d})^{\frac{1}{d-2}} > \vert l \vert$, hence it suffices to show that $(3-\frac{1}{d})^{\frac{1}{d-2}}\cdot(1-\frac{1}{d})>1$. Indeed, by elementary calculus, $(3-\frac{1}{d})(1-\frac{1}{d})^{d-2}>1$ for all $d \geq 3$, which gives $\vert z\cdot (z^{d-2} -2 ) \cdot (z-n)\vert> \vert l \vert$, hence by Rouch\'e's theorem, $f_n$ has $d-1$ roots of absolute value less than $\sqrt[d-2]{3-\frac{1}{d}}$. Therefore, $f_n$ has exactly $d-2$ roots with absolute values in the interval $(1, \sqrt[d-2]{3-\frac{1}{d}})$.

The last part of the lemma is obvious, since $x \cdot(x^{d-2}-2)\cdot(x-n)$ changes the sign at $0$ and at $n$ in the same way.
\end{proof}

\begin{lemma}\label{lem:irredodd}
Let $n\geq \vert l \vert +3$. Then $f_n$ is irreducible in $\mathbb{Q}[x]$ whenever $l$ is odd.
\end{lemma}
\begin{proof}
From Lemma \ref{lem:rootslocation} we know $\alpha_1^{(n)} > \vert l \vert$. Hence, $\alpha_1^{(n)}$ must be a conjugate of the only root of $f_n$ which is less than $1$ in absolute value. If $f_n$ would be reducible, then some product of the elements $\alpha_2^{(n)},\ldots,\alpha_{d-1}^{(n)}$ must be a divisor of $l$. But every such product lies strictly between $1$ and $3$. Since $2$ is no divisor of $l$ by assumption, $f_n$ is necessarily irreducible.
\end{proof}

\begin{lemma}\label{lem:Eisenstein}
Let $p$ be a prime and let $f=x^d + a_{d-1}x^{d-1}+\ldots+a_2 x^2 + a_1 x + a_0 \in \ZZ[x]$ such that $p\mid a_i$ for all $i\in \{0,\ldots,d-1\}$ and $p^2 \nmid a_2$. Then either $f$ has a divisor of degree $\leq 2$ or $f$ is irreducible.
\end{lemma}
\begin{proof}
This follows exactly as the classical Eisenstein criterion. Assume, that $f=g\cdot h$ where
\[
g(x)=x^r + g_{r-1}x^{r-1}+\ldots+g_0 \quad \text{ and } \quad h(x)=x^s+h_{s-1}x^{s-1}+\ldots+h_0 \in \ZZ[x]
\]
with $r,s\geq 3$. Since the reduction of $g\cdot h$ modulo $p$ is equal to $x^d \in \nicefrac{\ZZ}{p\ZZ}[x]$ and $\nicefrac{\ZZ}{p\ZZ}[x]$ is an integral domain, we know that each coefficient of $g$ and $h$ is divisible by $p$. It follows $p^2 \mid g_0h_2 + g_1 h_1 + g_2 h_0 =a_2$, which is a contradiction. 
\end{proof}

\begin{lemma}\label{lem:irredeven}
Let $n \geq \vert l \vert +3$ and $\vert l \vert$ both be even. Then $f_n$ is irreducible.
\end{lemma}
\begin{proof}
We first note that $f_n$ does not have a factor of degree $1$. Otherwise, some divisor $a$ of $l$ would be a root of $f_n$. But $\vert a (a-n)\vert \geq  n -1 \geq \vert l \vert +1$. Hence, in particular, $f_n(a)\neq 0$ for all $a\mid l$. It follows, that $f_n$ is irreducible for $d=3$. From now on we assume $d\geq 4$.

If $l$ and $n$ are even, then $f_n(x)=x(x^{d-2}-2)(x-n)+l=x^d -nx^{d-1} -2 x^2 +2n x +l$ is -- by Lemma \ref{lem:Eisenstein} -- irreducible if it does not have a factor of degree $2$. 

Since  $\alpha_1^{(n)}$ is larger than $\vert l \vert$ (which is the absolute value of product of all roots of $f_n$), it must be conjugate to $\alpha_d^{(n)}$ which is the only root of absolute value $\leq 1$. If $\alpha_d^{(n)}$ would be the only conjugate of $\alpha_1^{(n)}$, then $\alpha_1^{(n)}+\alpha_d^{(n)} \in \ZZ$. This is not possible by Lemma \ref{lem:rootslocation}. This means, that there is no factor of degree $2$, having $\alpha_1^{(n)}$ or $\alpha_d^{(n)}$ as a root. This proves that $f_n$ is irreducible for $d=4$. For $d\geq 5$ the only possibility of a divisor of degree $2$ is $x^2 - (\alpha_i^{(n)}+\alpha_j^{(n)})x + \alpha_i^{(n)}\alpha_j^{(n)}$, for $i\neq j\in \{2,\ldots,d-1\}$. By Lemma \ref{lem:rootslocation}, we have $\vert\alpha_i^{(n)}\alpha_j^{(n)}\vert >1$ and $\vert \alpha_i^{(n)}\alpha_j^{(n)} \vert< \sqrt[d-2]{3-\frac{1}{d}}^2 < 2$. Hence, such polynomial is not in $\ZZ[x]$. We conclude that $f_n$ does not have a factor of degree $\leq 2$ and therefore $f_n$ is irreducible.   
\end{proof}

\begin{thm}\label{thm:st}
Let $d\geq3$ and $l\in\ZZ\setminus\{\pm1,0\}$ such that $(d,l)\notin\{(3,2),(3,-2)\}$. Moreover, let $b_1,b_2,\ldots$ be the sequence from Proposition \ref{prop:assump} and $c\geq2$ be an integer with $c\neq2$ if $d\in\{3,4\}$. Then any root $\alpha$ of $f_{\vert l \vert^{b_c-1}}(x)=x(x^{d-2}-2)(x-\vert l \vert^{b_c-1})+l$ is an algebraic integer of degree $d$, norm $l$, and orbit size $c+2$.
\end{thm}
\begin{proof}
The cases we have to exclude, are those which violate assumption (ii) in Proposition \ref{prop:assump} or satisfy $\vert l^{b_c-1} \vert<\vert l\vert +3$.

In Lemmas \ref{lem:irredodd} and \ref{lem:irredeven}, we proved that $\alpha$ has degree $d$. Moreover, by Lemma \ref{lem:rootslocation}, $\alpha$ satisfies assumptions (i) and (ii) from Proposition \ref{prop:assump}. As usual we denote with $\alpha_1,\ldots,\alpha_d$ the full set of conjugates of $\alpha$. Then by Lemma \ref{lem:rootslocation}, we achieve
$\vert \alpha_d l^{b_c} \vert > \frac{\vert l \vert}{2} \geq 1$ and $\vert \alpha_1 \vert < \vert l^{b_{c}-1} \vert +1  \leq \vert l^{b_c}\vert$. 

Furthermore, we know $\vert \alpha_1\vert > \vert l \vert^{b_c-1}-1 \geq \vert l \vert^{b_{c-1}}$ and $\vert \alpha_d l^{b_{c-1}}\vert < \frac{\vert l \vert^{b_{c-1}+1}}{\vert l \vert^{b_c-1}}<1$. Again from Lemma \ref{lem:rootslocation} we also have $\vert \alpha_d l^{b_{c-2}}\vert < 1$ and $\vert\alpha_1\vert> l^{b_{c-2}}$, if $c \geq 3$. 

What we have shown is that in the notation from Proposition \ref{prop:assump}, we have $c(\alpha)=c$, and hence $\# \cO_M(\alpha)=c+2$.
\end{proof}

\begin{rmk}
A closed formula for the recursion $b_1,b_2,\ldots$ is $b_n=\frac{1}{d}((d-1)^n+(-1)^{n-1})$. So Theorem \ref{thm:st} is fairly effective.
\end{rmk}

\begin{cor}
For any triple $(d,l,k)$ of integers, with $d\geq3$, $l\notin\{\pm1,0\}$, and $1\leq k$, there are algebraic integers $\alpha$ with $[\mathbb{Q}(\alpha):\mathbb{Q}]=d$, $N(\alpha)=l$ and $\# \cO_M(\alpha)=k$.
\end{cor}
\begin{proof}
For $(3,2,k)$ and $(3,-2,k)$ this is due to Dubickas \cite{DubickasNrsMM} (note that 
he states the case $N(\alpha)=2$, but then $-\alpha$ does the job in the case of 
negative norm). Together with Theorem \ref{thm:st} and the examples at the beginning 
of this note, we conclude the corollary.
\end{proof}

\section{Behavior of degree $4$ units and proof of Theorem \ref{thm:degree4}}\label{sec:degree4}

In light of Theorem \ref{thm:1}, one might ask if arbitrarily long but finite orbits occur for algebraic units. In this section we will prove Theorem \ref{thm:degree4}, which states that the orbit size of an algebraic unit of degree 4 must be 1, 2, or $\infty$.

Let $\alpha$ be an algebraic unit of degree $4$. If $\alpha$ is a root of unity, a Pisot number, a Salem number or an inverse of such number we surely have $\# \cO_M(\alpha)\leq 2$. Hence, we may and will assume for the rest of this section that the conjugates of $\alpha$ satisfy
\[
\vert \alpha_1 \vert \geq \vert \alpha_2\vert > 1 > \vert \alpha_3 \vert \geq \vert \alpha_4 \vert.
\]
Denote the Galois group of $\mathbb{Q}(\alpha_1,\alpha_2,\alpha_3,\alpha_4)/\mathbb{Q}$ by $G_{\alpha}$. For any $\beta \in \mathbb{Q}(\alpha_1,\alpha_2,\alpha_3,\alpha_4)$ we denote the Galois orbit of $\beta$ by $G_\alpha\cdot \beta$.

Then $M(\alpha)=\pm \alpha_1\alpha_2$ and
\[
G_{\alpha}\cdot (\alpha_1 \alpha_2) \subseteq \{\alpha_1\alpha_2, \alpha_1\alpha_3,\alpha_1\alpha_4,\alpha_2\alpha_3,\alpha_2\alpha_4,\alpha_3\alpha_4 \}.
\]

\begin{lemma}\label{lem:14=1}
If $\vert \alpha_1\alpha_4\vert = 1$ or $\vert \alpha_1\alpha_3\vert = 1$, then 
we have either $\# \cO_M(\alpha)=2$ or $\# \cO_M(\alpha)=\infty$. 
\end{lemma}
\begin{proof}
If $\vert \alpha_1\alpha_4\vert=1$, then also $\vert \alpha_2\alpha_3\vert =1$, and if $\vert \alpha_1\alpha_3\vert=1$, then also $\vert \alpha_2\alpha_4\vert =1$. In both cases we see 
\begin{equation}\label{eq:iff}
\vert \alpha_1\vert =\vert \alpha_2\vert ~ \Longleftrightarrow ~ \vert \alpha_3\vert =\vert \alpha_4\vert. 
\end{equation}
We first assume that $\alpha_1\notin \mathbb{R}$. Then $\alpha_2=\overline{\alpha_1}$ and hence $\vert \alpha_1\vert =\vert \alpha_2\vert$. Obviously it is $M(\alpha_1)=\alpha_1\alpha_2$. By our assumptions and \eqref{eq:iff}, all values $\vert \alpha_1\alpha_3\vert$, $\vert\alpha_1\alpha_4\vert$, $\vert\alpha_2\alpha_3\vert$, $\vert\alpha_2\alpha_4\vert$, $\vert\alpha_3\alpha_4\vert$ are less or equal to $1$. Hence $M^{(2)}(\alpha_1)=M(\alpha_1\alpha_2)=\alpha_1\alpha_2$. Therefore, $\# \cO_M(\alpha_1)=2$.

If $\alpha_1\in\mathbb{R}$ and $\vert \alpha_1\vert =\vert \alpha_2\vert$, then $\alpha_2=-\alpha_1$ and $\alpha_4=-\alpha_3$. Hence, the only non-trivial Galois conjugate of $M(\alpha_1)=\alpha_1^2$ is $\alpha_3^2$ and lies inside the unit circle. Therefore, $M^{(2)}(\alpha_1) = \alpha_1^2$ and $\# \cO_M(\alpha_1)=2$.

From now on we assume that $\vert \alpha_1 \vert \neq \vert \alpha_2\vert$. Then, by \eqref{eq:iff}, we have 
\[
\vert \alpha_1 \vert > \vert \alpha_2\vert > 1 > \vert \alpha_3 \vert > \vert \alpha_4\vert
\]
and $\alpha_1$ must be totally real. Moreover, we see
\begin{equation}\label{eq:Modd}
\alpha_1^n, \alpha_2^n, \alpha_3^n, \alpha_4^n \quad \text{ are pairwise distinct for all } n\in\mathbb{N}, 
\end{equation}
and 
\begin{equation}\label{eq:Meven}
(\alpha_1\alpha_2)^n, (\alpha_3\alpha_4)^n, (\alpha_1\alpha_3)^n, (\alpha_2\alpha_4)^n \quad \text{ are pairwise distinct for all } n\in\mathbb{N}. 
\end{equation}
We notice, that in this situation it is not possible that $\vert \alpha_1 \alpha_3 \vert =1$, since otherwise $\vert \alpha_2 \alpha_4\vert <1$ which contradicts $1=\vert \alpha_1\alpha_2\alpha_3\alpha_4\vert$. Therefore, $\vert \alpha_1 \alpha_4 \vert =1$, and $\alpha_4 = \pm \alpha_1^{-1}$. It follows that also $\alpha_3=\pm\alpha_2^{-1}$. This gives natural constraints on the Galois group $G_{\alpha}$, namely
\[
G_{\alpha}\subseteq\{\id,(12)(34),(13)(24),(14)(23),(14),(23),(1342),(1243)\}\subseteq S_4.
\] 
In particular, since $G_{\alpha}$ is a transitive subgroup of $S_4$ with order divisible by $4$, $$G_{\alpha}=\{\id,(12)(34),(13)(24),(14)(23)\} \text{ or } \{\id , (1342),(14)(23),(1243)\} \subseteq G_{\alpha}.$$

In the first case, $G_{\alpha}\cdot (\alpha_1\alpha_2)=\{\alpha_1\alpha_2,\alpha_3\alpha_4\}$, which implies that $\alpha_1\alpha_2$ is a quadratic unit. Hence $\# \cO_M(\alpha)=\# \cO_M(\alpha_1\alpha_2)+1=2$.

In the second case, $G_{\alpha}\cdot (\alpha_1\alpha_2)=\{\alpha_1\alpha_2,\alpha_3\alpha_4,\alpha_1\alpha_3,\alpha_2\alpha_4\}$. Note that $\alpha_1\alpha_2$ is still of degree $4$ by \eqref{eq:Meven}. Hence $M^{(2)}(\alpha_1)=M(\alpha_1\alpha_2)=\pm\alpha_1^2\alpha_2\alpha_3= \alpha_1^2$. By \eqref{eq:Modd} it follows $M^{(3)}(\alpha) = M(\alpha_1^2) =(\alpha_1\alpha_2)^2 = M(\alpha)^2$. Now, by induction and \eqref{eq:Meven} and \eqref{eq:Modd}, it follows $M^{(n)}(\alpha_1)=\alpha_1^{2^n}$ for all even $n \in \mathbb{N}$. Hence $\# \cO_M(\alpha_1)=\infty$.
\end{proof}

From now on, we assume:
\begin{equation}\label{generalassumption}
\vert \alpha_1 \alpha_4 \vert \neq 1\neq \vert \alpha_1\alpha_3\vert.
\end{equation}

\begin{lemma}
Assuming \eqref{generalassumption}, if $\alpha_1^n = \alpha_2^n$ or $\alpha_3^n=\alpha_4^n$ for some $n\in\mathbb{N}$, then $\# \cO_M(\alpha_1)=2$. 
\end{lemma}
\begin{proof}
Let $\alpha_1^n=\alpha_2^n$ for some $n\in\mathbb{N}$. Then $\frac{\alpha_1}{\alpha_2}$ is a root of unity. Since none of the elements $\frac{\alpha_1}{\alpha_3}$, $\frac{\alpha_1}{\alpha_4}$, $\frac{\alpha_2}{\alpha_3}$, $\frac{\alpha_2}{\alpha_4}$, $\frac{\alpha_3}{\alpha_1}$, $\frac{\alpha_3}{\alpha_2}$, $\frac{\alpha_4}{\alpha_1}$, $\frac{\alpha_4}{\alpha_2}$ lies on the unit circle, we have $G_{\alpha}\cdot (\frac{\alpha_1}{\alpha_2})\subseteq \{\frac{\alpha_1}{\alpha_2},\frac{\alpha_2}{\alpha_1},\frac{\alpha_3}{\alpha_4},\frac{\alpha_4}{\alpha_3}\}$. Hence 
\[
G_{\alpha}\subseteq \{\id, (12),(12)(34),(13)(24),(14)(23),(1324),(1423)\}.
\]
This implies $M^{(2)}(\alpha_1)=M(\pm \alpha_1\alpha_2)= \pm \alpha_1\alpha_2 = M(\alpha_1)$, and hence $\# \cO_M(\alpha_1)=2$. The same proof applies if $\alpha_3^n=\alpha_4^n$.
\end{proof}

\begin{lemma}\label{lem:properties2}
Assuming \eqref{generalassumption} and $\# \cO_M(\alpha_1)>2$, then
\begin{enumerate}[(a)]
\item $\vert \alpha_1 \alpha_2 \vert >1$, $\vert \alpha_1\alpha_3\vert >1$.
\item $\vert \alpha_3 \alpha_4 \vert <1$, $\vert \alpha_2\alpha_4\vert <1$.
\item one of the values $\vert \alpha_1\alpha_4\vert$ and $\vert \alpha_2\alpha_3\vert$ is $<1$ and the other is $>1$.
\item $\alpha_1^n$, $\alpha_2^n$, $\alpha_3^n$, $\alpha_4^n$ are pairwise distinct for all $n\in\mathbb{N}$.
\item $(\alpha_1\alpha_2)^n$, $(\alpha_3\alpha_4)^n$, $(\alpha_1\alpha_3)^n$, $(\alpha_2\alpha_4)^n$ are pairwise distinct for all  $n\in\mathbb{N}$. 
\end{enumerate}
\end{lemma}
\begin{proof}
Obviously $\vert \alpha_1\alpha_2 \vert >1$ and $\vert \alpha_3\alpha_4\vert <1$. Moreover, $1\neq\vert \alpha_1\alpha_3 \vert \geq \vert \alpha_2\alpha_4\vert$ and $ \vert \alpha_1\alpha_3 \vert \cdot \vert \alpha_2\alpha_4\vert =1$. This means $\vert\alpha_1\alpha_3\vert >1$ and $\vert \alpha_2\alpha_4\vert <1$, proving parts (a) and (b).

Since $\vert \alpha_1\alpha_4\vert \cdot \vert \alpha_2\alpha_3\vert=1$ and $\vert \alpha_1\alpha_4\vert \neq 1$, part (c) follows.

The elements $\alpha_1$ and $\alpha_2$ lie outside the unit circle, and $\alpha_3$ and $\alpha_4$ lie inside or on the unit circle. Hence, the only possibilities for (d) to fail are $\alpha_1^n=\alpha_2^n$ or $\alpha_3^n=\alpha_4^n$ for some $n\in\mathbb{N}$. By the previous lemma, both implies $\# \cO_M(\alpha_1)=2$, which is excluded by our assumptions.

Part (e) follows immediately from (a), (b) and (d).
\end{proof}

\begin{lemma}\label{lem:infty}
If $M^{(3)}(\alpha_1)=M(\alpha_1)^2$ and $\# \cO_M(\alpha_1)>2$, then $\# \cO_M(\alpha_1)=\infty$.
\end{lemma}
\begin{proof}
This is true if assumption \eqref{generalassumption} is not satisfied, by Lemma \ref{lem:14=1}. If we assume \eqref{generalassumption}, then by Lemma \ref{lem:properties2} (d) and (e), we are in the same situation as at the end of the proof of Lemma \ref{lem:14=1}. Hence, an easy induction proves the claim.
\end{proof}

We now complete the proof of the statement that $\# \cO_M(\alpha_1)\in\{1,2,\infty\}$. 
It suffices to prove this under the assumption \eqref{generalassumption}. From 
now on we assume $\# \cO_M(\alpha)>2$ and show that this implies $\# \cO_M(\alpha)=\infty$. 
By Lemma \ref{lem:properties2}, we have
\begin{align}
M^{(2)}(\alpha) &\in\{\pm\alpha_1^2\alpha_2\alpha_3,\pm\alpha_1^3\alpha_2\alpha_3\alpha_4,\pm\alpha_1^2\alpha_2^2\alpha_3^2,\pm\alpha_1^2\alpha_2\alpha_4,\pm\alpha_1\alpha_2^2\alpha_3\}  \nonumber \\ &=\{\pm\frac{\alpha_1}{\alpha_4}, \pm\alpha_1^2,\pm\frac{1}{\alpha_4^2},\pm\frac{\alpha_1}{\alpha_3},\pm\frac{\alpha_2}{\alpha_4} \}
\end{align}
In two of these cases the orbit of $\alpha$ can be determined immediately:
\begin{itemize}
\item If $M^{(2)}(\alpha)=\pm \alpha_1^2$, then (since we have $\# \cO_M(\alpha)>2$) it is $\alpha_1^n \neq \alpha_2^n$ for all $n\in\mathbb{N}$. Hence $M^{(3)}(\alpha)=M(\alpha)^2$ which implies $\# \cO_M(\alpha)=\infty$.
\item Similarly, if $M^{(2)}(\alpha)=\pm \frac{1}{\alpha_4^2}$, then  (since $\# \cO_M(\alpha)>2$) it is $\alpha_3^n \neq \alpha_4^n$ for all $n\in\mathbb{N}$. Hence $M^{(3)}(\alpha)=M(\alpha_4^2)=M(\alpha)^2$ and again $\# \cO_M(\alpha)=\infty$. 
\end{itemize}

We now study the other three cases.

\subsection{The case $M^{(2)}(\alpha)=\pm\frac{\alpha_1}{\alpha_4}$} 

This case occurs if $\alpha_1\alpha_3\in G_{\alpha}\cdot(\alpha_1\alpha_2)$, and
\begin{itemize}
\item $\vert \alpha_1\alpha_4\vert >1$ but $\alpha_1\alpha_4 \not\in G_{\alpha}\cdot(\alpha_1\alpha_2)$, or
\item $\vert \alpha_2\alpha_3\vert >1$ but $\alpha_2\alpha_3 \not\in G_{\alpha}\cdot(\alpha_1\alpha_2)$.
\end{itemize}
 
In both cases the only possibilities for $G_{\alpha}$ are the following copies 
of the cyclic group $C_4$ and the dihedral group $D_8$:
\begin{itemize}
\item[(I)] $C_4=\{\id, (1342), (14)(23), (1243) \}$, or
\item[(II)] $D_8=\{\id, (1243),(14)(23),(1342),(12)(34),(13)(24),(14),(23) \}$.
\end{itemize} 

In both cases a full set of conjugates of $\frac{\alpha_1}{\alpha_4}$ is $\{\frac{\alpha_1}{\alpha_4}, \frac{\alpha_3}{\alpha_2},\frac{\alpha_4}{\alpha_1}, \frac{\alpha_2}{\alpha_3}  \}$. It follows 
\[
M^{(3)}(\alpha)=M(\frac{\alpha_1}{\alpha_4})=\pm \frac{\alpha_1}{\alpha_4}\cdot \frac{\alpha_2}{\alpha_3}=   (\alpha_1\alpha_2)^2 = M(\alpha)^2
\]
Hence, by Lemma \ref{lem:infty} we have $\# \cO_M(\alpha)=\infty$.

\subsection{The case $M^{(2)}(\alpha)=\pm\frac{\alpha_1}{\alpha_3}$} 

This case occurs if $\alpha_1\alpha_3 \not\in G_{\alpha}\cdot(\alpha_1\alpha_2)$, and
$\vert \alpha_1\alpha_4\vert >1$, and $\alpha_1\alpha_4 \in G_{\alpha}\cdot(\alpha_1\alpha_2)$.

Hence, the only possibilities for $G_{\alpha}$ are the following copies of the cyclic group $C_4$ and the dihedral group $D_8$:
\begin{itemize}
\item[(I)] $C_4=\{\id, (1234), (13)(24), (1432) \}$, or
\item[(II)] $D_8=\{\id, (1234),(13)(24),(1432),(12)(34),(14)(23),(13),(24) \}$.
\end{itemize} 

In both cases a full set of conjugates of $\frac{\alpha_1}{\alpha_3}$ is $\{\frac{\alpha_1}{\alpha_3}, \frac{\alpha_2}{\alpha_4},\frac{\alpha_3}{\alpha_1}, \frac{\alpha_4}{\alpha_2}  \}$. It follows 
\[
M^{(3)}(\alpha)=M(\frac{\alpha_1}{\alpha_3})=\pm \frac{\alpha_1}{\alpha_3}\cdot \frac{\alpha_2}{\alpha_4}=   (\alpha_1\alpha_2)^2 = M(\alpha)^2
\]
Hence, again we have $\# \cO_M(\alpha)=\infty$ by Lemma \ref{lem:infty}.

\subsection{The case $M^{(2)}(\alpha)=\pm\frac{\alpha_2}{\alpha_4}$} 

This case occurs if $\alpha_1\alpha_3 \not\in G_{\alpha}\cdot(\alpha_1\alpha_2)$, and
$\vert \alpha_2\alpha_3\vert >1$, and $\alpha_2\alpha_3 \in G_{\alpha}\cdot(\alpha_1\alpha_2)$.

Hence, the only possibilities for $G_{\alpha}$ are the following copies of the cyclic group $C_4$ and the dihedral group $D_8$:
\begin{itemize}
\item[(I)] $C_4=\{\id, (1234), (13)(24), (1432) \}$, or
\item[(II)] $D_8=\{\id, (1234),(13)(24),(1432),(12)(34),(14)(23),(13),(24) \}$.
\end{itemize} 

In both cases a full set of conjugates of $\frac{\alpha_2}{\alpha_4}$ is $\{\frac{\alpha_2}{\alpha_4}, \frac{\alpha_3}{\alpha_1},\frac{\alpha_4}{\alpha_2}, \frac{\alpha_1}{\alpha_3}  \}$. It follows 
\[
M^{(3)}(\alpha)=M(\frac{\alpha_2}{\alpha_4})=\pm \frac{\alpha_2}{\alpha_4}\cdot \frac{\alpha_1}{\alpha_3}=  \pm (\alpha_1\alpha_2)^2 = M(\alpha)^2
\]
Hence, also in this case we have $\# \cO_M(\alpha)=\infty$. 

This concludes the proof of Theorem \ref{thm:degree4}. We now prove Corollary \ref{cor:to-thm-degree4}:
\begin{proof}[Proof of Corollary \ref{cor:to-thm-degree4}]
Let $\alpha$ be an algebraic unit of degree $4$. We set \[a_n=\log(M^{(n)}(\alpha))\] for all $n\in\mathbb{N}$. If $\# \cO_M(\alpha)\leq 2$, then $a_{n+1}=a_n$ for all $n\in\mathbb{N}$. If $\# \cO_M(\alpha)=\infty$, then Theorem \ref{thm:degree4} tells us $a_3=2a_1$. Moreover, $M^{(4)}(\alpha)=M(M^{(3)}(\alpha))=M(M(\alpha)^2)=M(M(\alpha))^2=M^{(2)}(\alpha)^2$. Hence, $a_4=2a_2$, and by induction we find $a_{n+1}=2a_{n-1}$, proving the claim.
\end{proof}

\section{Symmetric and alternating Galois groups}\label{sec:Sd-or-Ad}

In this section we will prove Theorem \ref{thm:Sd-or-Ad}. We know that $\# \cO_M(\al) \in \{1,2,\infty\}$ whenever $\al$ is an algebraic unit of degree $\leq 4$. (We note in passing that the orbit size for units of degree less than $4$ is trivially $1$ or $2$.) So we assume from now on that $\alpha$ is an algebraic unit with $[\mathbb{Q}(\alpha):\mathbb{Q}]=d\geq 5$. Denote by $G_{\alpha}$ the Galois group of the Galois closure of $\mathbb{Q}(\alpha)$. We assume that $G_{\alpha}$ contains a subgroup isomorphic to $A_d$, so $G_{\alpha}$ is either the full symmetric group or the alternating group. Every self-reciprocal polynomial admits natural restrictions on which permutations of the zeros are given by field automorphisms. Hence, $\alpha$ cannot be conjugated to $\pm$ a Salem number (see \cite{Salem-Galoisgroup} for more precise statements on the structure of the Galois group $G_\al$, when $\alpha$ is a Salem number). If one of $\pm\al^{\pm1}$ is conjugated to a Pisot number, then surely $\# \cO_M(\al)\in\{1,2\}$. Hence, we assume from now on that none of $\pm\al^{\pm 1}$ is conjugated to a Pisot number. 

Hence, if we denote by $\alpha_1,\ldots,\alpha_d$ the Galois conjugates of $\alpha$, we assume  
\begin{align}\label{eq:conjugatedistribution}
\vert \alpha_1\vert \geq \vert \alpha_2\vert \geq \ldots \geq \vert \alpha_r\vert > 1 \geq \vert \alpha_{r+1} \vert \geq \ldots \geq \vert \alpha_d \vert,\\ \text{where } r\in\{2,\ldots,d-2\} \text{ and } 1> \vert \alpha_{d-1} \vert. \nonumber
\end{align}
We identify $G_{\alpha}$ with a subgroup of $S_d$, by the action on the indices of $\alpha_1,\ldots,\alpha_d$. In particular, for any $\sigma \in A_d$ and any $f_1,\ldots,f_d \in \mathbb{Z}$ the element
\[
\sigma\cdot (\alpha_1^{f_1}\cdot\ldots\cdot\alpha_d^{f_d}):=\alpha_{\sigma(1)}^{f_1}\cdot\ldots\cdot\alpha_{\sigma(d)}^{f_d}
\] 
is a Galois conjugate of $\alpha_1^{f_1}\cdot\ldots\cdot\alpha_d^{f_d}$.

\begin{lemma}\label{lem:newGal}
Let $i,j,k,l\in\{1,\ldots,d\}$ be pairwise distinct, and let $f_1,\ldots,f_d \in \mathbb{Z}$. Then
\begin{enumerate}[(a)]
\item $(i,j,k)\cdot (\alpha_1^{f_1}\cdots \alpha_d^{f_d}) = \alpha_1^{f_1}\cdots \alpha_d^{f_d} ~ \Longleftrightarrow ~ f_i=f_j=f_k$.
\item $(i,j)(k,l)\cdot (\alpha_1^{f_1}\cdots \alpha_d^{f_d}) = \alpha_1^{f_1}\cdots \alpha_d^{f_d} ~ \Longleftrightarrow ~ f_i=f_j \text{ and } f_k=f_l$.
\end{enumerate}
\end{lemma}
\begin{proof}
In both statements, the implication $\Longleftarrow$ is trivial. Lets start with the other implication in (a). It is
\begin{align*}
(i,j,k)\cdot (\alpha_1^{f_1}\cdots \alpha_d^{f_d}) = \alpha_1^{f_1}\cdots \alpha_d^{f_d} ~ \Longrightarrow ~ \alpha_j^{f_i - f_j}\cdot \alpha_k^{f_j-f_k}\cdot \alpha_i^{f_k-f_i}=1  
\end{align*}
Since $d\geq 5$, we may choose two conjugates of $\alpha$ not among $\alpha_i,\alpha_j,\alpha_k$ -- say $\alpha_p$ and $\alpha_q$. Since $G_{\alpha}$ contains $A_d$, the elements $(i,j)(p,q)$, $(i,k)(p,q)$, $(j,k)(p,q)$, $(i,j,k)$, and $(i,k,j)$ are all contained in $G_{\alpha}$. Applying these automorphisms to $\alpha_j^{f_i - f_j}\cdot \alpha_k^{f_j-f_k}\cdot \alpha_i^{f_k-f_i}=1$, yields
\begin{align*}
\alpha_j^{f_i - f_j}\cdot \alpha_k^{f_j-f_k}\cdot \alpha_i^{f_k-f_i} &=1 =\alpha_j^{f_i - f_j}\cdot \alpha_i^{f_j-f_k}\cdot \alpha_k^{f_k-f_i}\\
\alpha_i^{f_i - f_j}\cdot \alpha_j^{f_j-f_k}\cdot \alpha_k^{f_k-f_i} &=1 =\alpha_k^{f_i - f_j}\cdot \alpha_j^{f_j-f_k}\cdot \alpha_i^{f_k-f_i}\\
\alpha_k^{f_i - f_j}\cdot \alpha_i^{f_j-f_k}\cdot \alpha_j^{f_k-f_i} &=1 =\alpha_i^{f_i - f_j}\cdot \alpha_k^{f_j-f_k}\cdot \alpha_j^{f_k-f_i}.
\end{align*}
Hence
\[
\left( \frac{\alpha_i}{\alpha_k} \right)^{2f_k-f_i-f_j} =1, \qquad \left( \frac{\alpha_i}{\alpha_k} \right)^{2f_i-f_j-f_k} =1, \quad \text{and} \quad \left( \frac{\alpha_i}{\alpha_k} \right)^{2f_j-f_k-f_i} =1.
\]
But $\frac{\alpha_i}{\alpha_k}$ is no root of unity, since it is a Galois 
conjugate of $\frac{\alpha_1}{\alpha_d}$, which lies outside the unit 
circle. It follows $2f_k-f_i-f_j=2f_i-f_j-f_k=2f_j-f_k-f_i=0$, and hence 
$f_i=f_j=f_k$. This proves part (a).

Part (b) follows similarly: $(i,j)(k,l)\cdot (\alpha_1^{f_1}\cdots \alpha_d^{f_d})=\alpha_1^{f_1}\cdots \alpha_d^{f_d}$ implies
\[
\alpha_i^{f_j}\cdot \alpha_j^{f_i}\cdot \alpha_k^{f_l}\cdot \alpha_l^{f_k}= \alpha_i^{f_i}\cdot \alpha_j^{f_j}\cdot \alpha_k^{f_k}\cdot \alpha_l^{f_l}.
\]
Without loss of generality, we assume $f_j\geq f_i$ and $f_k\geq f_l$. Using 
that $(i,l)(j,k)$ is an element of $G_{\alpha}$, we get
\[
\left( \frac{\alpha_j}{\alpha_i} \right)^{f_j-f_i} = \left( 
\frac{\alpha_k}{\alpha_l} \right)^{f_k-f_l} \quad \text{ and } \quad \left( 
\frac{\alpha_k}{\alpha_l} \right)^{f_j-f_i} = \left( \frac{\alpha_j}{\alpha_i} 
\right)^{f_k-f_l}.
\]
Multiplying both equations yields
\[
\left( \frac{\alpha_j}{\alpha_i} \right)^{(f_j - f_i)+(f_k-f_l)}=\left( \frac{\alpha_k}{\alpha_l} \right)^{(f_j - f_i)+(f_k-f_l)},
\]
and hence 
\[
\left( \frac{\alpha_j\cdot \alpha_l}{\alpha_i\cdot \alpha_k} \right)^{(f_j - f_i)+(f_k-f_l)}=1.
\]
Again, $\frac{\alpha_j\cdot \alpha_l}{\alpha_i\cdot \alpha_k}$ is a Galois conjugate of $\frac{\alpha_1\cdot \alpha_2}{\alpha_{d-1}\cdot \alpha_d}$, which lies outside the unit circle, and hence is not a root of unity. Therefore $(f_j - f_i)+(f_k-f_l)=0$. Since $f_j\geq f_i$ und $f_k\geq f_l$, it follows $f_j=f_i$ and $f_k = f_l$, proving the lemma.
\end{proof}

\begin{lemma}\label{lem:Mnreal}
Let $f_1,\ldots,f_d$ be pairwise distinct integers. Then $[\mathbb{Q}(\alpha_1^{f_1}\cdots\alpha_d^{f_d}):\mathbb{Q}]=\# G_{\alpha}$.
\end{lemma}
\begin{proof}
The proof is essentially the same as the proof of part (1) in Theorem 1.1 from 
\cite{Amoroso}. Assume there is a $\sigma^{-1}\in G_{\alpha}\subseteq S_d$ such 
that $\alpha_1^{f_1}\cdots\alpha_d^{f_d} = 
\sigma^{-1}\cdot(\alpha_1^{f_1}\cdots\alpha_d^{f_d})$. Then
\begin{equation}\label{eqn:lem9-1}
1 = \alpha_1^{f_1-f_{\sigma(1)}}\cdots \alpha_d^{f_d-f_{\sigma(d)}}.
\end{equation}
If $\sigma$ is an odd permutation, then $G_\al=S_d$, then it was already proven by Smyth (see Lemma 1 of \cite{SmythAddMultRels}) that $f_i=f_{\sigma(i)}$ for all $i$, hence that $\sigma=\id$. If $\sigma$ is an even permutation, then by repeated application of Lemma \ref{lem:newGal} to equation \eqref{eqn:lem9-1} above, this is only possible if $f_i-f_{\sigma(i)}$ is the same integer for all $i\in\{1,\ldots,d\}$, say $f_i-f_{\sigma(i)}=k$.

Since $\sigma^{d!}=\id$, it follows
\[
f_1 = k+ f_{\sigma(1)} = 2k + f_{\sigma^2(1)} = \ldots = d!\cdot k + 
f_{\sigma^{d!} (1)}=d!\cdot k + f_1,
\]
and hence $k=0$. Therefore we have $f_i=f_{\sigma(i)}$ for all 
$i\in\{1,\ldots,d\}$. But by assumption the integers $f_1,\ldots,f_d$ are 
pairwise distinct, hence, we must again have $\sigma=\id$. Since in either 
case, $\sigma=\id$, this means that the images of 
$\alpha_1^{f_1}\cdots\alpha_d^{f_d}$ are distinct under each non-identity 
element of $G_\al$, so 
$[\mathbb{Q}(\alpha_1^{f_1}\cdots\alpha_d^{f_d}):\mathbb{Q}]=\# G_{\alpha} $. 
\end{proof}

\begin{prop}\label{prop:expdiffer}
Let $M^{(n)}(\alpha)=\alpha_1^{e_1}\cdot \ldots\cdot e_d^{e_d}$ such that the exponents $e_1,\ldots,e_d$ are pairwise distinct. 
Then $M^{(n+1)}(\alpha)> M^{(n)}(\alpha)$.
\end{prop}
\begin{proof}
We denote by $Z_3$ the set of $3$-cycles in $G_{\alpha}\subseteq S_d$. For any $k\in\{1,\ldots,d\}$, the number of $3$-cycles which fix $k$ is equal to $\frac{(d-1)(d-2)(d-3)}{3}$. For any pair $k\neq k'\in\{1,\ldots,d\}$, the number of $3$-cycles sending $k$ to $k'$ is $(d-2)$. Therefore,
\begin{multline}\label{eq:3cycleprod}
\left\vert \prod_{\tau \in Z_3} \tau \cdot M^{(n)}(\alpha) \right\vert \\
= \left\vert \alpha_1^{\frac{(d-1)(d-2)(d-3)}{3}e_1 + (d-2)\sum_{k\neq 1} e_k} \cdot \ldots \cdot \alpha_d^{\frac{(d-1)(d-2)(d-3)}{3}e_d + (d-2)\sum_{k\neq d} e_k}  \right\vert.
\end{multline}
Since $\alpha$ is an algebraic unit, we have $\prod_{j=1}^d \alpha_j^{\sum_{k=1}^d e_k}=\pm 1$. Hence, the value in \eqref{eq:3cycleprod} is equal to
\[
\left\vert  \prod_{j=1}^d \alpha_j^{\left(\frac{(d-1)(d-2)(d-3)}{3}-(d-2)\right)e_j} \right\vert = M^{(n)}(\alpha)^{\frac{(d-1)(d-2)(d-3)}{3}-(d-2)}>M^{(n)}(\alpha).
\]
The last inequality follows from our general hypothesis that $d\geq 5$. Since $e_1,\ldots,e_d$ are assumed to be pairwise distinct, it follows from Lemma \ref{lem:Mnreal} that the factors $\tau\cdot M^{(n)}(\alpha)$ in \eqref{eq:3cycleprod} are also pairwise distinct conjugates of $M^{(n)}(\alpha)$. In particular 
\[
M^{(n+1)}(\alpha) = M(M^{n}(\alpha)) \geq \left\vert \prod_{\tau \in Z_3} \tau \cdot M^{(n)}(\alpha) \right\vert > M^{(n)}(\alpha)
\]
which is what we needed to prove.
\end{proof}

\begin{lemma}\label{lem:exprop}
Let $n\in\mathbb{N}$ and let $M^{(n)}(\alpha)=\alpha_1^{e_1}\cdots\alpha_d^{e_d}$. Then we have:
\begin{enumerate}[(a)]
\item \label{enum:a}$e_i \geq e_{i+1}$ for all but at most one $i\in\{1,\ldots,d-1\}$.
\item \label{enum:b}If $e_i< e_{i+1}$ for some $i\in\{2,\ldots,d-1\}$, then $e_{i-1} > e_{i+1}$.
\item \label{enum:c}If $e_i<e_{i+1}$ for some $i\in\{1,\ldots,d-2\}$, then $e_{i} > e_{i+2}$.
\item \label{enum:d}If $e_i < e_{i+1}$ for some $i\in\{1,\ldots,d-1\}$, then 
\[
e_1 > e_2 > \cdots > e_{i-1} > e_{i+1} > e_i >  e_{i+2} > e_{i+3} > \cdots > e_d.
\]
\end{enumerate}
\end{lemma}
\begin{proof}
It is known that $M^{(n)}(\alpha)$ is a Perron number, which means that 
$M^{(n)}(\alpha)$ does not have a Galois conjugate of the same or larger modulus (cf. 
\cite{DubickasNrsMM} for this and other properties of values of the Mahler measure). 
This fact will be used several times in the following proof.

To prove \eqref{enum:a}, we have two cases: there are three distinct elements $1\leq i<j<k\leq d$ such that $e_i < e_j < e_k$, or else there exist $1\leq i<j<k<l\leq d$ such that $e_i < e_j$ and $e_k < e_l$. Assume first that there are three distinct elements $1\leq i<j<k\leq d$ such that $e_i < e_j < e_k$. Recall that by definition we have $\vert \alpha_i\vert \geq \vert \alpha_j\vert \geq \vert \alpha_k\vert$. Therefore $\vert \alpha_k \vert^{e_k -e_i} \leq \vert \alpha_j \vert^{e_k-e_i}$, which implies
\begin{align*}
 & \underbrace{\vert \alpha_i \vert^{e_i-e_j}}_{\leq \vert \alpha_j \vert^{e_i-e_j}} \cdot \vert \alpha_j\vert^{e_j-e_k} \cdot \vert \alpha_k\vert^{e_k-e_i} \leq \vert \alpha_j \vert^{e_i-e_k} \cdot \vert \alpha_k\vert^{e_k - e_i} \leq 1 \\
 \Longrightarrow ~ & \vert \alpha_i^{e_i} \cdot \alpha_j^{e_j} \cdot \alpha_k^{e_k} \vert \leq \vert \alpha_i^{e_j} \cdot \alpha_j^{e_k} \cdot \alpha_k^{e_i} \vert\\
 \Longrightarrow ~ & \vert M^{(n)}(\alpha) \vert \leq \vert (i,k,j)\cdot M^{(n)}(\alpha) \vert.
\end{align*}
By Lemma \ref{lem:newGal}, $(i,j,k)\cdot M^{(n)}(\alpha) \neq  M^{(n)}(\alpha)$ is a Galois conjugate of $M^{(n)}(\alpha)$. This contradicts the fact that $M^{(n)}(\alpha)$ is a Perron number. In particular, it is not possible that $e_i> e_{i+1} > e_{i+2}$ for any $i\in\{1,\ldots,d-2\}$.

Now assume that we have $1\leq i<j<k<l \leq d$ such that $e_i < e_j$ and $e_k < e_l$. Then $\vert \alpha_i \vert \geq \vert \alpha_j \vert$ and $\vert \alpha_k \vert \geq \vert \alpha_l \vert$ imply
\[
\vert \alpha_i \vert^{e_j - e_i} \cdot \vert \alpha_k \vert^{e_l -e_k} \geq \vert \alpha_j \vert^{e_j - e_i} \cdot \vert \alpha_l \vert^{e_l -e_k},
\]
and hence
\[
\vert \alpha_i^{e_i} \cdot \alpha_j^{e_j} \cdot \alpha_k^{e_k} \cdot \alpha_l^{e_l} \vert \leq \vert \alpha_i^{e_j} \cdot \alpha_j^{e_i} \cdot \alpha_k^{e_l} \cdot \alpha_l^{e_k} \vert .
\]
This, however, is equivalent to $\vert M^{(n)}(\alpha)\vert \leq \vert (i,j)(k,l)\cdot M^{(n)}(\alpha) \vert$, which is not possible by Lemma \ref{lem:newGal}, since $M^{(n)}(\alpha)$ is a Perron number. This proves part \eqref{enum:a} of the lemma.

In order to prove part \eqref{enum:b}, we assume for the sake of contradiction that $e_i<e_{i+1}$ but $e_{i-1}\leq e_{i+1}$ for some $i\in\{2,\ldots,d-1\}$. By part \eqref{enum:a}, since we already have $e_i<e_{i+1}$, we know that $e_{i-1}\geq e_i$. We have 
\begin{align*}
& (i-1,i,i+1)\cdot \vert \alpha _{i-1}\vert ^{e_{i-1}}\vert \alpha_{i}\vert ^{e_{i}}\vert \alpha_{i+1}\vert ^{e_{i+1}}\\
 &=\vert \alpha _{i}\vert ^{e_{i-1}}\vert \alpha_{i+1}\vert ^{e_{i}}\vert \alpha_{i-1}\vert ^{e_{i+1}}
\end{align*} Now,
 \begin{align*}
 & \vert \alpha _{i-1}\vert ^{e_{i-1}-e_{i+1}}\vert \alpha_{i}\vert ^{e_{i}-e_{i-1}}\vert \alpha_{i+1}\vert ^{e_{i+1}-e_{i}}\\
 &=\vert \alpha _{i-1}\vert ^{e_{i-1}-e_{i+1}}\vert \alpha_{i}\vert ^{e_{i}-e_{i-1}}\vert \alpha_{i+1}\vert ^{e_{i+1}-e_{i-1}}\vert \alpha_{i+1}\vert ^{e_{i-1}-e_{i}}\\
 &\leq \vert \alpha _{i-1}\vert ^{e_{i-1}-e_{i+1}}\vert \alpha_{i}\vert ^{e_{i}-e_{i-1}}\vert \alpha_{i-1}\vert ^{e_{i+1}-e_{i-1}}\vert \alpha_{i}\vert ^{e_{i-1}-e_{i}}\\
 &= 1
\end{align*}
Therefore,
\[\vert M^{(n)}(\alpha) \vert \leq \vert (i-1,i,i+1)\cdot M^{(n)}(\alpha) \vert,\] giving a contradiction to $M^{(n)}(\al)$ being a Perron number.

Similarly, if $e_i< e_{i+1}$ and $e_i\leq e_{i+2}$, then we know by (a) that $e_{i+1}\geq e_{i+2}$. This implies that $\vert M^{(n)}(\alpha) \vert \leq \vert (i,i+2,i+1)\cdot M^{(n)}(\alpha)\vert$. This proves part \eqref{enum:c}.

So far we have proven that if $e_i<e_{i+1}$ for some $i\in\{1,\ldots,d-1\}$, then we have 
\[
e_1 \geq e_2 \geq \ldots \geq e_{i-1} > e_{i+1} > e_i >  e_{i+2} \geq e_{i+3} \geq \ldots \geq e_d.
\]
We need to show that all of the above inequalities are strict. Assume that this is not the case, and that $e_k=e_{k+1}$. Then $k,k+1,i,i+1$ must be pairwise distinct. It follows, that $\vert (i,i+1)(k,k+1)\cdot M^{(n)}(\alpha) \vert =  \vert (i,i+1)\cdot M^{(n)}(\alpha) \vert > \vert M^{(n)}(\alpha) \vert$, which is a contradiction. The last inequality just follows from the fact that $\vert \alpha_i \vert^{e_{i+1}} \cdot \vert \alpha_{i+1}\vert^{e_i} > \vert \alpha_i \vert^{e_{i}} \cdot \vert \alpha_{i+1}\vert^{e_{i+1}}$.
\end{proof}

\begin{lemma}\label{lem:exp=1}
Let $f_1\geq f_2\geq \ldots\geq f_k\geq 0$ be integers, with $f_1 \geq 1$, and let $a_1\geq a_2\geq\ldots\geq a_d >0$ be real numbers such that $\prod_{i=1}^k a_i>1$. Then $\prod_{i=1}^k a_i^{f_i} >1$.
\end{lemma}
\begin{proof}
We prove the statement by induction on $k$, where the base case $k=1$ is trivial. Now assume that the statement is true for $k$ and that there are real numbers $a_1\geq\ldots\geq a_{k+1} >0$, with $\prod_{i=1}^{k+1}a_i > 1$, and integers $f_1\geq \ldots\geq f_{k+1} \geq 0$, with $f_1\geq 1$. If $f_1=f_{k+1}$, then the claim follows immediately. Hence, we assume $f_1 > f_{k+1}$. Set $f_i'=f_i-f_{k+1}$ for all $i\in\{1,\ldots,k+1\}$. Then
\[
f_1'\geq f_2'\geq \ldots f_k'\geq f_{k+1}'=0 \text{ and } f_1'\geq 1.
\]
Moreover, $\prod_{i=1}^k a_i$ is either greater than or equal to $\prod_{i=1}^{k+1} a_i >1$ (if $a_{k+1}\leq 1$), or it is a product of real numbers $>1$. Hence, our induction hypothesis states $\prod_{i=1}^k a_i^{f_i'} >1$. This implies
\[
\prod_{i=1}^{k+1} a_i^{f_i} = \underbrace{\left(\prod_{i=1}^{k+1} a_i \right)^{f_{k+1}}}_{\geq 1} \cdot \left( \prod_{i=1}^k a_i^{f_i'} \right) > 1,
\]
proving the lemma. 
\end{proof}

\begin{prop}\label{prop:ordered}
Let $M^{(n)}(\alpha)=\alpha_1^{e_1}\cdots\alpha_d^{e_d}$. If $e_{i+1}\leq  e_i$ for all $i\in\{1,\ldots,d-1\}$, then $M^{(n+1)}(\alpha)> M^{(n)}(\alpha)$.
\end{prop}
\begin{proof}
We show that $M^{(n)}(\alpha)$ has a non-trivial Galois conjugate outside the unit circle. This immediately implies the claim.

Since $\alpha$ is an algebraic unit, we may assume that $e_d=0$. Note however, that this uses our assumption $e_{i+1}\leq e_{i}$ for all $i$. We set
\[
s:=\max\{i\in\{1,\ldots,d\}\vert e_i\neq 0\}.
\]
By Proposition \ref{prop:expdiffer} we may assume that we have $e_i=e_{i+1}$ for some $i\in\{1,\ldots,d-1\}$. This $i$ is not equal to $s$, since $e_s\neq 0 =e_{s+1}$ by definition. If $i\notin\{s-1,s+1\}$, then $(i,i+1)(s,s+1)\cdot M^{(n)}(\alpha)=\alpha_1^{e_1}\cdots \alpha_{s-1}^{e_{s-1}}\alpha_s^{e_{s+1}}\alpha_{s+1}^{e_s}$. If $i=s-1$, then $(s-1,s+1,s)\cdot M^{(n)}(\alpha)=\alpha_1^{e_1}\cdots \alpha_{s-2}^{e_{s-2}}\alpha_{s-1}^{e_{s}}\alpha_s^{e_{s+1}}\alpha_{s+1}^{e_{s-1}}=\alpha_1^{e_1}\cdots \alpha_{s-1}^{e_{s-1}}\alpha_s^{e_{s+1}}\alpha_{s+1}^{e_s}$. If finally $i=s+1$, then $(s,s+1,s+2)\cdot 
 M^{(n)}(\alpha)=\alpha_1^{e_1}\cdots \alpha_{s-1}^{e_{s-1}}\alpha_s^{e_{s+2}}\alpha_{s+1}^{e_{s}}\alpha_{s+2}^{e_{s+1}}=\alpha_1^{e_1}\cdots \alpha_{s-1}^{e_{s-1}}\alpha_s^{e_{s+1}}\alpha_{s+1}^{e_s}$.
 
Since $e_{s+1}=0$, we see that in any case 
\begin{equation}\label{eq:conj}
\alpha_1^{e_1}\cdots \alpha_{s-1}^{e_{s-1}}\alpha_{s+1}^{e_s}  \text{ is a non-trivial Galois conjugate of } M^{(n)}(\alpha).
\end{equation}
We will prove that this Galois conjugate lies outside the unit circle. Again we distinguish several cases.

If $s\leq r-1$, then all of the elements $\alpha_1,\ldots,\alpha_{s+1}$ lie outside the unit circle. Hence $\vert \alpha_1\cdots \alpha_{s-1}\alpha_{s+1} \vert >1$.

If $s\geq r+1$, then  $\vert \alpha_1\cdots \alpha_{s-1}\alpha_{s+1} \vert = \vert \alpha_s \alpha_{s+2}\cdots\alpha_{d}\vert^{-1} >1$, since all of $\alpha_s,\ldots,\alpha_d$ lie inside the closed unit disc and $\vert \alpha_d\vert < 1$.

Lastly, we consider the case $2\leq s=r \leq d-2$. Then surely $\vert \alpha_1 \cdots \alpha_{r-1} \vert \geq \vert \alpha_r \vert$ and $\vert \alpha_{r+1}\vert \geq \vert \alpha_{r+2}\cdots\alpha_{d} \vert$, where the first inequality is strict whenever $r\neq 2$, and the second inequality is strict whenever $r\neq d-2$. By our general assumption it is $d\geq 5$ and hence $\vert \alpha_1\cdots\alpha_{r-1}\alpha_{r+1}\vert > \vert \alpha_r \alpha_{r+2} \cdots\alpha_d\vert$. Since the product of all $\alpha_i$ is $\pm1$, it follows $\vert \alpha_1\cdots \alpha_{s-1}\alpha_{s+1} \vert >1$. 

Hence, in any case we have $\vert \alpha_1\vert \cdots \vert\alpha_{s-1}\vert \cdot \vert\alpha_{s+1} \vert >1$. From our assumption $e_1\geq \ldots \geq e_d$ it follows by Lemma \ref{lem:exp=1} that $\vert \alpha_1^{e_1}\cdots \alpha_{s-1}^{e_{s-1}}\alpha_{s+1}^{e_s}\vert >1$. Therefore, $M^{(n)}(\alpha)$ has a non-trivial Galois conjugate outside the unit circle (see \eqref{eq:conj}). Hence $M^{(n+1)}(\alpha)=M(M^{(n)}(\alpha))> M^{(n)}(\alpha)$.
\end{proof}

We are now ready to prove Theorem \ref{thm:Sd-or-Ad}.
\begin{proof}[Proof of Theorem \ref{thm:Sd-or-Ad}]
As stated at the beginning of this section, we may assume that $d\geq 5$, and that the elements $\pm \al^{\pm1}$ are neither conjugates of a Pisot, nor a Salem number.
Hence, we may assume that the hypothesis \eqref{eq:conjugatedistribution} is met. Let $n\in\mathbb{N}$ be arbitrary. Then for some $e_1,\ldots,e_d \in\mathbb{N}_0$, we have $M^{(n)}(\alpha) =\alpha_1^{e_1}\cdots\alpha_d^{e_d}$. We have seen in Lemma \ref{lem:exprop}, that one of the following statements applies:
\begin{enumerate}[(i)]
\item $e_1\geq e_2 \geq \ldots \geq e_d$, or
\item the integers $e_1,\ldots,e_d$ are pairwise distinct.
\end{enumerate} 
In case (i), we have $M^{(n+1)}(\alpha) > M^{(n)}(\alpha)$ by Proposition \ref{prop:ordered}. In case (ii), we have $M^{(n+1)}(\alpha) > M^{(n)}(\alpha)$ by Proposition \ref{prop:expdiffer}. Hence $\# \cO_M(\alpha)=\infty$.
\end{proof}

\section{Arbitrarily large finite orbit size for units of degree 12}\label{sec:deg-12}
Let $d=4k$, with an integer $k\geq 3$. Now, we will show that there exist algebraic units of degree $d$ with arbitrarily large orbit size, proving Theorem \ref{thm:deg-12}.
\begin{proof}[Proof of Theorem \ref{thm:deg-12}]
Let $\alpha_1, \beta_1$ be positive real algebraic units satisfying:
 \begin{enumerate}
 \item $[\mathbb{Q}(\beta_1):\mathbb{Q}]=2$, $\beta_1>1$,
 \item $\alpha_1$ is a Salem number of degree $2k$.
 \item The fields $\mathbb{Q}(\alpha_1)$ and $\mathbb{Q}(\beta_1)$ are linearly disjoint.
 
 \end{enumerate}
 For any $k\geq 3$ we can indeed find such $\alpha_1$ and $\beta_1$. Since there are Salem numbers of any even degree $\geq 4$ we find an appropriate $\alpha_1$. Now, we take any prime $p$ which is unrammified in $\mathbb{Q}(\alpha_1)$, and let $\beta_1>1$ be an algebraic unit in $\mathbb{Q}(\sqrt{p})$. Note that if the above conditions are met by $\alpha_1$ and $\beta_1$, then they are met by $\alpha_1^{\ell}$ and $\beta_1^{\ell'}$, for any $\ell, \ell' \in \mathbb{N}$. 

We denote the conjugates of $\alpha_1$ by $\alpha_2, \cdots, \alpha_{2k}$, with $\alpha_{2k}=\alpha_1^{-1}$, and the conjugate of $\beta_1$ is $\beta_2=\beta_1^{-1}$. Note that $\alpha_2,\cdots,\alpha_{2k-1}$ all lie on the unit circle.
By assumption (3) the element $\alpha_1\beta_1$ has degree $4k$ and a full set of Galois conjugates of $\alpha_1\beta_1$ is given by
\[
\{\alpha_i\beta_j : (i,j)\in\{1,\ldots,2k\}\times \{1,2\}\}.
\]

 There are two cases. First, if $\beta_1>\alpha_1$, then $\vert\alpha_i\beta_1\vert>1$ for all $i\in\{1,\cdots, 2k\}$ and $\vert \alpha_i\beta_2\vert <\vert\alpha_i\alpha_6\vert\leq 1$ for all $i\in\{1,\cdots, 2k\}$, hence, 
\begin{equation}\label{eqn:deg12-1}
M(\alpha_1\beta_1)=\left|\prod_{n=1}^{2k}\alpha_i\beta_1\right|=\beta_1^{2k}
\end{equation}

For the second case, if $\beta_1<\alpha_1$, then 
\[
\vert\alpha_i\beta_1\vert>1 \iff i\in \{1, \cdots, 2k-1\},\, \text{and} \, \,\vert\alpha_i\beta_2\vert>1 \iff i=1.
\]
Therefore
\begin{equation}\label{eqn:deg12-2}
M(\alpha_1\beta_1)=\vert\alpha_1\beta_1\vert\cdot\left|\prod_{n=2}^{2k-1}\alpha_i\beta_1\right|\cdot\vert\alpha_1\beta_2\vert=\alpha_1^2\beta_1^{2k-2}.
\end{equation}
 We now construct an algebraic unit of degree $4k$ of finite orbit size $>S$. Let $\ell\in \mathbb{N}$ be such that $(\alpha_1^\ell)^{2^S}>\beta_1^{(2k-2)^S}$. Then by \eqref{eqn:deg12-2}, we have $M(\alpha_1^\ell\beta_1)=(\alpha_1^\ell)^2\beta_1^{2k-2}$, $M^{(2)}(\alpha_1^\ell\beta_1)=M((\alpha_1^\ell)^2)(\beta_1^{2k-2}))=(\alpha_1^\ell)^{2^2}\beta_1^{(2k-2)^2},\cdots, M^{(S)}(\alpha_1^\ell\beta_1)=(\alpha_1^\ell)^{2^S}\beta_1^{(2k-2)^S}$. Hence, the orbit size of $\alpha_1^\ell\beta_1$ is greater than $S$. However, there exists $S'>S$ such that $(\alpha_1^\ell)^{2^{S'}}<\beta_1^{(2k-2)^{S'}}$. Assume that $S'$ is minimal with this property. Then we have 
 \[
 M^{(S'+1)}(\alpha_1^\ell\beta_1)=(\alpha_1^\ell)^{2^{S'}}\beta_1^{(2k-2)^{S'}} \stackrel{\eqref{eqn:deg12-1}}{=}(\beta_1^{4^{S'}})^{2k},
 \]
 which is of degree $2$. Therefore, the orbit size of $ \alpha_1^\ell\beta_1 $ is $ S'+2>S $.
\end{proof}

\bibliographystyle{abbrv} 
\bibliography{bib}

\end{document}